\documentclass[10pt,bezier]{article}
\usepackage{amsmath,amssymb,amsfonts}

\newtheorem{prethm}{{\bf Theorem}}

\newenvironment{thm}{\begin{prethm}{\hspace{-0.5
               em}{\bf.}}}{\end{prethm}}
\newtheorem{prepro}{{\bf Theorem}}

\newtheorem{precor}{{\bf Corollary}}

\newtheorem{preconj}{{\bf Conjecture}}

\newtheorem{preremark}{{\bf Remark}}

\newenvironment{remark}{\begin{preremark}\em{\hspace{-0.5
               em}{\bf.}}}{\end{preremark}}
\newtheorem{prelem}{{\bf Lemma}}

\newenvironment{lem}{\begin{prelem}{\hspace{-0.5
               em}{\bf.}}}{\end{prelem}}
\newtheorem{preproof}{{\bf Proof.}}

\newenvironment{proof}[1]{\begin{preproof}{\rm
               #1}\hfill{$\Box$}}{\end{preproof}}

\usepackage{graphicx}

\date {}

\title{\bf\Large  A note on zero-sum $5$-flows in regular  graphs}

\author{{}\\{}\\
{\normalsize{\sc S. Akbari\,${}^{a, b}$, N. Ghareghani\,${}^{ d, b}$,
G. B. Khosrovshahi\,${}^{b, e}$, S. Zare\,${}^{c}$}}
\\{\footnotesize{\it ${}^{a}$Department of
Mathematical Sciences, Sharif University of Technology}}
\\{\footnotesize{\it ${}^{b}$School of Mathematics, Institute for Research in Fundamental Sciences (IPM)}}
\\{\footnotesize{\it ${}^{c}$ Department of
Mathematical Sciences, Amirkabir University of Technology}}
\\{\footnotesize{\it ${}^{d}$ Department of
Mathematical Sciences, K.N. Toosi University of Technology}}
\\{\footnotesize{\it ${}^{e}$ Department of
Mathematical Sciences, University of Tehran}}
\thanks{{\it E-mail addresses}: $\mathsf{s\_akbari@sharif.edu}$
(S. Akbari), $\mathsf{ghareghani@ipm.ir}$ (N. Ghareghani), $\mathsf{
rezagbk@ipm.ir }$ (G.B. Khosrovshahi),
$\mathsf{sa\_zare\_f@yahoo.com}$ (S. Zare).}
\thanks {AMS (2000) { \it Subject classification}: 05C21, 05C22.}
\thanks {Corresponding author: S. Akbari.}}
\date{}
\begin{document}
\maketitle
\begin{abstract}
Let $G$ be a graph. A {\it zero-sum flow} of $G$ is an assignment
of non-zero real numbers to the edges such that the sum of the
values of all edges incident with each vertex is zero. Let $k$ be
a natural number. A {\it zero-sum $k$-flow} is a flow with values
from the set \mbox{$\{\pm 1,\ldots ,\pm(k-1)\}$}. It has been
conjectured that every $r$-regular graph, $r\geq 3$, admits a zero-sum
$5$-flow. In this paper we give an affirmative answer to this
conjecture, except for  $r=5$.
\end{abstract}

\date{}

\vspace{9mm} \noindent{\bf\large 1. Introduction}\vspace{5mm}

Nowhere-zero flows on graphs were introduced by Tutte \cite{tut} in
1949 and since then have been extensively studied by many
authors.  A great deal of research in the area has been motivated
by Tutte's 5-Flow Conjecture  which states that every $2$-edge connected
graph can have its edges directed and labeled by integers from
$\{1, 2, 3, 4\}$ in such a way that Kirchhoff's current law is
satisfied at each vertex. In 1983,  Bouchet \cite{bouchet}
generalized this concept to bidirected graphs. A \emph{bidirected graph} $G$ is a graph with vertex set $V(G)$
and edge set $E(G)$ such that each edge is oriented as one of
the four possibilities: 
\begin{picture}(45,5)(0,-2)
\thicklines
\put(15,0){\vector(-1,0){4}}\put(27,0){\vector(1,0){4}}
\thinlines \put(0,0){\line(1,0){40}} \put(0,0){\circle*{4}}
\put(40,0){\circle*{4}}
\end{picture},
\begin{picture}(45,5)(0,-2)
\thicklines \put(15,0){\vector(1,0){4}}\put(27,0){\vector(1,0){4}}
\thinlines \put(0,0){\line(1,0){40}} \put(0,0){\circle*{4}}
\put(40,0){\circle*{4}}
\end{picture},
\begin{picture}(45,5)(0,-2)
\thicklines
\put(15,0){\vector(1,0){4}}\put(27,0){\vector(-1,0){4}}
\thinlines \put(0,0){\line(1,0){40}} \put(0,0){\circle*{4}}
\put(40,0){\circle*{4}}
\end{picture},
\begin{picture}(45,5)(0,-2)
\thicklines
\put(15,0){\vector(-1,0){4}}\put(27,0){\vector(-1,0){4}}
\thinlines \put(0,0){\line(1,0){40}} \put(0,0){\circle*{4}}
\put(40,0){\circle*{4}}
\end{picture}.
 Let $G$ be a bidirected graph. For every $v \in V(G)$, the
set of all edges with tails (respectively, heads) at $v$ is denoted by
$E^+(v)$ (respectively, $E^-(v)$). The function $f:E(G) \longrightarrow
\mathbb{R}$ is a \emph{bidirected flow} of $G$ if for every $v
\in V(G)$, we have
$$\sum_{e \in E^+(v)} f(e) = \sum_{e \in E^-(v)} f(e).$$ If $f$ takes its values from the set \mbox{$\{\pm 1,\ldots ,\pm(k-1)\}$}, then it is called a \emph{nowhere-zero bidirected $k$-flow}.

Consequently, Bouchet proposed the following interesting  conjecture. \vspace{2mm}

\noindent {\bf Bouchet's Conjecture.} {\rm \cite{bouchet, xui}}
\textit{Every bidirected graph that has a nowhere-zero bidirected
flow admits a nowhere-zero bidirected $6$-flow.}

Bouchet proved that his conjecture is true if $6$ is replaced by
$216$. Then Zyka reduced $216$ to $30$  \cite{zyk}.

Let $G$ be a graph. A {\it zero-sum flow} for $G$ is an assignment of non-zero real numbers
to the edges  such that the sum of the values of all edges
incident with each vertex is zero. Let $k$ be a natural number. A {\it zero-sum $k$-flow} is a
flow with values from the set \mbox{$\{\pm 1,\ldots
,\pm(k-1)\}$}. The following conjecture was posed on the zero-sum flows in graphs.

\noindent{\bf Zero-Sum Conjecture (ZSC).} {\rm \cite{zerosum1}}   \textit{If $G$ is a graph
with a zero-sum flow, then $G$ admits a zero-sum $6$-flow.}
\vspace{2mm}

The following conjecture is an improved version of ZSC for regular graphs.

\noindent{\bf  Conjecture A.}   {\rm \cite{d}}  \textit{ Every $r$-regular graph ($r\geq 3$)
admits a zero-sum $5$-flow.}

Recently, in  connection with this conjecture the following two theorems were proved.

\begin{thm}{ \em \cite{zerosum1}} \label{evenreg}
 \textit{ Let $r$ be an even integer with $r\geq 4$. Then every
$r$-regular graph has a zero-sum $3$-flow.}
\end{thm}

\begin{thm}{ \em \cite{d}}\label{3reg}  \textit{ Let $G$ be an $r$-regular graph.
If  $r$ is divisible by $3$, then  $G$ has a zero-sum $5$-flow.}
\end{thm}

\begin{remark} {There are some regular graphs with no zero-sum $4$-flow. To see this
consider the graph given in Figure $1$. To the contrary assume this the graph has a zero-sum $4$-flow. Since
the sum of the values of all edges incident with each vertex is zero, for every $v \in V(G)$,
$-2$ or $2$ should  appear in the neighborhood of $v$. On the other hand two numbers with absolute
value $2$ can not appear in the neighborhood of a vertex. So all edges of $G$ with values $\pm 2$
form a perfect matching. But by celebrated Tutte's  Theorem \cite[p.76] {bondy}, $G$ has no  perfect matching, a contradiction.}
\end{remark}

\vspace{3cm} \includegraphics{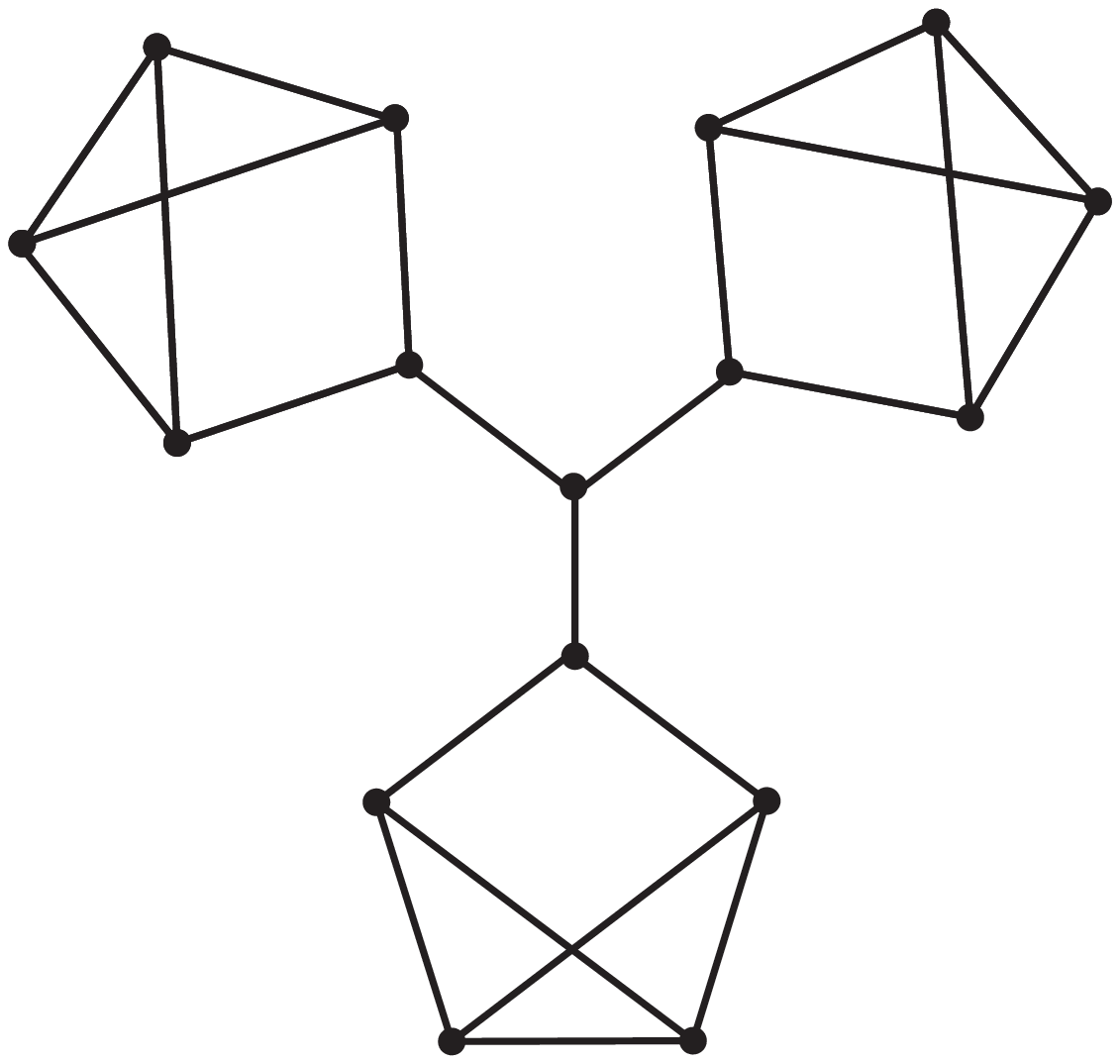} \vspace{4cm} \vspace{-2.5cm}$$\textmd{Figure $1$. A $3$-regular graph with no zero-sum $4$-flow}$$

In $2010$, the following result was proved.

\begin{thm} {\em \cite{d}}\label{bou}{ \textit Bouchet's Conjecture
and ZSC are equivalent.}
\end{thm}

Motivated by Bouchet's Conjecture and along with Theorem \ref{bou}
we focused our attention to establish the Conjecture A. We show that except  $r = 5$, Conjecture A is true.

\vspace{9mm} \noindent{\bf\large 2. The Main Result }\vspace{5mm}

In this section we prove that every $r$-regular graph, $r\geq 3$,
$r\neq 5$, admits a zero-sum $5$-flow. Before establishing our main
result we need some notations and definitions.

Let $G$ be a  finite and undirected graphs with vertex set $V(G)$ and edge set $E(G)$,
where multiple edges and loops are admissible. A {\it
$k$-regular} graph is  a graph where each vertex is of degree $k$.
A subgraph $F$ of a graph  $G$, is a
{\it factor} of $G$ if $F$ is a spanning subgraph of $G$. If a
factor $F$ has all of its degrees equal to $k$, it is called a
{\it$k$-factor}.  Thus a
$2$-factor is a disjoint union of finitely many cycles that cover
all the vertices of $G$. A {\it $k$-factorization} of $G$ is a partition of  the edges of $G$ into disjoint $k$-factors.
For integers $a$ and $b$, $1 \leq a \leq b$, an {\it $[a, b]$-factor} of
$G$ is defined  to be a factor $F$ of $G$ such that
$a \leq d_F (v) \leq b$, for every  $v \in V (G)$.  For any vertex $v\in V(G)$, let
$N_{G}(v)=\{\,u \in V(G)\,|\,uv\in E(G)\,\}$.

The following two theorems are also needed.

\begin{thm}{ \em \cite{handbook}}\label{peterson}
\textit{Every $2k$-regular multigraph admits a $2$-factorozation.}
\end{thm}

\begin{thm}{ \em \cite{kano}}\label{kano}
\textit{Let $r \geq 3$ be an odd integer and let $k$ be an integer such
that $1 \leq k \leq \frac{2r}{3}$. Then every $r$-regular graph
has a $[k-1, k]$-factor each component of which is regular.}
\end{thm}

\begin{lem}\label{qr}
\textit{Let  $G$ be an $r$-regular graph. Then
for every even integer $q$, $ 2r \leq q \leq 4r$, there exists
a function $f: E(G) \rightarrow \{2, 3, 4\}$ such that for every
$u \in V(G)$, $\sum_{v \in N_G(u)}f(uv)=q$.}
\end{lem}
\begin{proof}
{First assume that  $r$ is an odd integer. For every edge $e=uv$, we add a new edge $e'=uv$ to the graph $G$ and call the resultant graph by $G'$.
Clearly, $G'$ is a $2r$-regular
 multigraph. By Theorem \ref{peterson}, $G'$ admits   a $2$-factorization  with $2$-factors $F_1, \ldots, F_{r}$.
Now,  for every $e \in F_i$, $1 \leq i \leq r$, we define a  function $g: E(G') \rightarrow \{1,2\}$ as
follows:

$$g(e)=\left\{
  \begin{array}{ll}
     2, & \hbox{$1 \leq i \leq  \frac{q-2r}{2}$;}\\
 1, & \hbox{$ \frac{q-2r}{2}<i$.}

  \end{array}
\right.$$

Therefore   for each  $v
\in V(G')$, $\sum_{v \in N_{G'}(u)}g(uv)=q$. Now,  define a function $f: E(G) \rightarrow \{2,3,4\}$ such that
for every $e=uv \in E(G)$, $f(e)= g(e)+g(e')$, where $e'=uv$ in $G'$.
Then for every $u \in V(G)$,    $\sum_{v \in N_G(u)}f(uv)=q$, as desired.

Now, let  $r$ be an even integer. Since $G$ is an
$r$-regular graph, by Theorem \ref{peterson}, $G$ admits a
$2$-factorization  with $2$-factors $F_1, \ldots, F_{\frac{r}{2}}$. Now, for every $e\in F_i$,  $1 \leq i \leq\frac{ r}{2}$, we define a function
$f: E(G) \rightarrow \{2, 3, 4\}$ as follows:

$$f(e)=\left\{
   \begin{array}{ll}
     4, & \hbox{$1 \leq i \leq \lfloor \frac{q-2r}{4}\rfloor $;} \\
     3, & \hbox{$\lfloor \frac{q-2r}{4} \rfloor < i\leq  \lceil \frac{q-2r}{4} \rceil$;} \\
     2, & \hbox{$ \lceil \frac{q-2r}{4} \rceil<i.$}
   \end{array}
 \right.$$

It is not hard to verify that for every $u \in V(G)$,   $\sum_{v
\in N_G(u)}f(uv)=q$, as desired. }
\end{proof}

Now, we are in a position to prove our main theorem.

\begin{thm}
\textit{Let $r \geq 3$ and $r \neq 5$. Then
every $r$-regular graph has a zero-sum $5$-flow.}
\end{thm}

\begin{proof}
{
First we prove the theorem for  $r=7$. Let $G$ be a $7$-regular
graph. Then by Theorem \ref{kano}, $G$ has a $[3, 4]$-factor, say $H$,
whose  components are regular. Let $H_1$ be the union of the
$3$-regular components of $H$ and let $H_2$ be the union of
$4$-regular components of $H$.  By Theorem \ref{peterson}, $H_2$
can be decomposed into two $2$-factors $H'_{2}$ and $H''_{2}$.
Assign $1$ and $2$ to all edges of $H'_{2}$ and $H''_{2}$, respectively.
 By Lemma \ref{qr}, there exists a function $f: E(H_1)
\rightarrow \{2, 3, 4\}$ such that for every $u \in V(H_1)$,
$\sum_{v \in N_{H_1}(u)}f(uv)=8$. Now, assign $-2$ to  every edge in $E(G)\setminus E(H)$ and we are done.

Now, let $r\geq 9$ be an odd integer. By Theorem
\ref{kano},  for every $k$, $k \leq \frac{2r}{3}$, $G$ has a
$[k-1, k]$-factor whose components are regular. Let $k=\lfloor
\frac{2r}{3} \rfloor$,  $k'= r-k$, and  $H$ be a $[k-1,
k]$-factor of $G$ such that $H_1$ be the union of $(k-1)$-regular
subgraph of $H$ and $H_2=H \setminus H_1$. It can be easily
checked that $k \leq 2k' \leq 2k-4$. Hence  by Lemma \ref{qr}, there
exists a function $f:E(H_1) \longrightarrow \{2, 3, 4\}$
such that for every $u \in V(H_1)$, $\sum_{v \in
N_{H_1}(u)}f(uv)= 4k'+4$. Also by Lemma \ref{qr}, there exists a
function $f:E(H_2) \longrightarrow \{2, 3, 4\}$  such that
for every $v \in V(H_2)$,
 $\sum_{v \in N_{H_2}(u)}f(uv)= 4k'$.
 Finally assign $-4$ to  every edge of $E(G)\setminus E(H)$. Now, by Theorem \ref{evenreg} and
 Theorem \ref{3reg}   the proof is complete. }
\end{proof}

\noindent {\bf Acknowledgements.}  The authors are indebted to the
School of Mathematics, Institute for Research in Fundamental
Sciences (IPM) for the support. The research of the first author and the
second author were in parts supported by grantns from IPM (No.
88050212) and (No.  88050042), respectively.

\providecommand{\bysame}{\leavevmode\hbox
to3em{\hrulefill}\thinspace}

\end{document}